\documentclass{au}
\usepackage{amsmath,amssymb,epic,url}

\theoremstyle{plain}
\newtheorem{theorem}{Theorem}
\newtheorem{lemma}[theorem]{Lemma}

\theoremstyle{definition}
\newtheorem{definition}[theorem]{Definition}

\newcommand{\Z}{\mathbb{Z}}

\newcommand{\ciX}[2]{\put(#1){\circle{20}\put(-7.6,-2.6){#2}}}

\begin{document}

\title[Submonoids of groups, and relation algebras]%
{Submonoids of groups, and group-representability of\\
restricted relation algebras}

\author{George M. Bergman}
\email{gbergman@math.berkeley.edu}
\urladdr{https://math.berkeley.edu/~gbergman}
\address{Department of Mathematics\\University of California\\
Berkeley, CA 94720-3840, USA}

\thanks{Work partly supported by NSF contract MCS~80-02317.
(The author apologizes for the long delay in publishing this
1981 result.)\\
\hspace*{1.6em}
\url{http://arxiv.org/abs/1702.06088}\,.
After publication of this note, updates, errata, related references
etc., if found, will be recorded at
\url{http://math.berkeley.edu/~gbergman/papers/}
}

\dedicatory{To the memory of Bjarni J\'{o}nsson}



\begin{abstract}
Marek Kuczma asked in 1980 whether for every positive integer $n$,
there exists a subsemigroup $M$ of a group $G$, such that
$G$ is equal to the $n$-fold product
$M\,M^{-1} M\,M^{-1} \dots\,M^{(-1)^{n-1}}$, but not to any
proper initial subproduct of this product.
We answer his question affirmatively,
and prove a more general result on representing
a certain sort of relation algebra by subsets of a group.
We also sketch several variants of the latter result.
\end{abstract}
\maketitle

\section{Introduction}\label{S.intro}

M.\,Kuczma \cite[Problem P190, p.\,304]{Tagung}
raised the question quoted in
the Abstract for $n=3$ in particular, and also for arbitrary $n$.
The case $n=3$ was answered affirmatively by an example
of W.\,Benz~\cite[Remark P190S1, p.\,305]{Tagung}.
We sketch in Section~\ref{S.u_d} a construction that
works for all $n$, then prove in Section~\ref{S.main} a general result,
Theorem~\ref{T.main}, from which, as we show in~Section~\ref{SS.u_d},
the asserted behavior of that example follows.

In Section~\ref{S.variants} we look at some variants (and further
possible variants) of Theorem~\ref{T.main}.
In particular, in Section~\ref{SS.BJ} we note a class of operations
on binary relations, described in B.\,J\'{o}nsson's
survey paper~\cite{BJ}, which that theorem can be extended
to cover.

\section{Sketch of the construction answering Kuczma's question}\label{S.u_d}

If $G$ is a group of permutations of a set $X$, we shall write
elements of $G$ to the right of elements of $X$, and compose
them accordingly.

Given a positive integer $n$, let $X_1,\dots,X_n$ be disjoint
\emph{infinite} sets, let $X=X_1\cup\,\dots\,\cup X_n$, and
let $G$ be the group of those permutations $g$ of $X$
such that for all but finitely many $x\in X$,
the element $xg$ lies in the same set $X_i$ as does $x$.
Note that since the $X_i$ are infinite, the (finite) number of
elements carried into a given $X_i$ from elsewhere by a given $g\in G$
need not equal the (finite) number moved out of $X_i$.
Indeed, given any finite set of elements of $X$,
and any assignment of a destination-set $X_i$ for each of them,
one can construct a $g\in G$ which realizes these movements,
and keeps all other elements in their original sets.
(In achieving such a rearrangement,
if the number of newcomers assigned to some $X_i$
is not equal to the number of elements specified to leave $X_i$,
then infinitely many elements can be moved within $X_i$,
as in ``Hilbert's Hotel'' \cite[p.\,17]{123oo}, to accommodate
these relocations.)

Now let $M\subseteq G$ be the submonoid consisting of
those $g$ which involve no movement of elements from one
$X_i$ to a different one, except for transitions from even-indexed sets
$X_{2i}$ to \emph{adjacent} odd-indexed sets, $X_{2i-1}$ and
$X_{2i+1}$, as suggested by the picture~\eqref{d.u_d} below.
(For concreteness, $n$ is there assumed even.)
Dotted arrows show directed paths along
which finitely many elements may move.
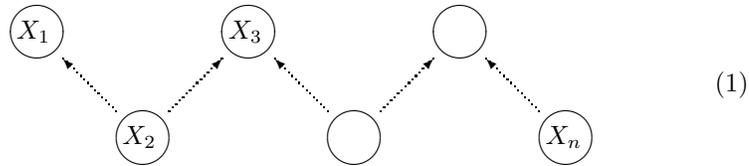
\begin{equation}\begin{minipage}[c]{23pc}\label{d.u_d}
\begin{center}
\begin{picture}(220,70)
\ciX{10,55}{$X_1$}
\put(20,45){\dottedline{2}(0,0)(20,-20)
   {\vector(-1,1){0}}}
\ciX{50,15}{$X_2$}
\put(80,45){\dottedline{2}(0,0)(-20,-20)
   {\vector(1,1){0}}}
\ciX{90,55}{$X_3$}
\put(100,45){\dottedline{2}(0,0)(20,-20)
   {\vector(-1,1){0}}}
\ciX{130,15}{}
\put(160,45){\dottedline{2}(0,0)(-20,-20)
   {\vector(1,1){0}}}
\ciX{170,55}{}
\put(180,45){\dottedline{2}(0,0)(20,-20)
   {\vector(-1,1){0}}}
\ciX{210,15}{$X_n$}
\end{picture}
\end{center}
\end{minipage}\end{equation}
Clearly, $M^{-1}$ is the submonoid of $G$
corresponding to the same diagram with the dotted arrows reversed.

One finds that using $n$ elements of $M$
and $M^{-1}$, acting alternately, starting with an element of $M$,
one can achieve the action
of any element of $G$; but that one cannot in general represent
such an element using a shorter product of this sort.
The reader might think this through now, or wait for the details
that will be given in Sections~\ref{S.defs}-\ref{S.appl}.
In convincing oneself that $n$ factors do suffice,
it can be helpful to think of first finding a sequence of
$n$ elements, alternately from $M$ and from $M^{-1}$, that together
bring every
element into the correct $X_i$, then modifying the last of these
factors by a permutation which preserves each of the sets $X_i$,
and simply rearranges elements within those sets as needed.
One may ask why $n-1$ factors $M M^{-1}\dots$ do not suffice,
since there are only $n-1$ links in~\eqref{d.u_d}.
The reason is that if an element of $X_1$ is to travel to $X_n$,
it has to wait for the second factor, $M^{-1}$, to begin
this journey, since the initial factor $M$ does not move elements
out of $X_1$.
(And if $n$ is odd, the same applies to an element traveling
from $X_n$ to $X_1$.)

Assuming the above handwaving filled in, as will be done
below, this construction
answers Kuczma's question affirmatively for general~$n$.

\section{Sketch of the motivation for a more general result}\label{S.abstr}

The definition of $M$ above can be thought of as based
on a partial ordering $\preccurlyeq$ on the index set $\{1,\dots,n\}$,
under which the only comparability relations say that
each even element is less than the two adjacent odd elements
(cf.~\eqref{d.u_d}).
The monoid $M$ embodies this partial ordering in
allowing elements to move from $X_i$ to $X_j$ if $i\preccurlyeq j$.
The monoid $M^{-1}$ similarly embodies the reverse partial ordering.

What do sets such as $M\,M^{-1}$ look like?
Application of an element $g h^{-1}$ $(g,h\in M)$ allows elements to
move from $X_i$ to $X_j$ if $(i,j)$ belongs to
$\{(i,j):(\exists\,k)\ i\preccurlyeq k\succcurlyeq j\}$,
the \emph{composite}
of the binary relations $\preccurlyeq$ and $\succcurlyeq$ on $X$.
So we can think of the construction of the preceding section
as starting with a binary relation $\preccurlyeq$ on the
set $I=\{i,\dots,n\}$, obtaining from it a subset $M$
of $G$, and applying the operations of elementwise inversion of subsets
of $G$, and multiplication of such subsets, to get
subsets similarly determined by various composites of
the binary relation $\preccurlyeq$ and its inverse ${\succcurlyeq}$.
The maximal binary relation,
$I\times I$, is reached as a certain $n$-fold composite of
$\preccurlyeq$ and $\succcurlyeq$, but not as a
proper initial subcomposite thereof, and
this turns out to imply that we get all of $G$ as
the corresponding $n$-fold product of $M$ and $M^{-1}$,
but not as a proper initial subproduct thereof.

Not every binary relation on $I$ determines
a nonempty subset of $G$, since members of $G$ by definition
leave all but finitely many
elements of each $X_i$ in $X_i$; so let us restrict the relations
on $I$ from which we form such subsets to be reflexive, i.e.,
to contain the diagonal relation $\{(i,i):i\in I\}$.
Not all of the resulting subsets will be submonoids:
the fact that the $M$ discussed above was closed
under multiplication is a consequence of the fact
that the relation $\preccurlyeq$, a partial ordering, is transitive.
But, for example, $M\,M^{-1}$, corresponding to the
relation $\preccurlyeq\circ\succcurlyeq$, is not a submonoid if $n>2$.

So far, we have been very sketchy.
The next two sections develop in detail the
result suggested by these observations: that the system of all
reflexive binary relations on \emph{any} set $I$ is mirrored in a
system of subsets of a corresponding group $G$, in a way that
respects certain operations on these two families.
In Section~\ref{SS.u_d} we verify that this implies
the asserted properties of the example of Section~\ref{S.u_d} above.

\section{Down to business: definitions}\label{S.defs}

In the remainder of this note, we shall distinguish
between an algebra $A$ (in the sense of universal algebra)
and its underlying set, which we denote $|A|$.
We continue to write
permutations to the right of the elements they are applied to,
and to compose them accordingly.
We shall similarly compose binary relations like functions
written to the right of their arguments.

For $I$ a set, logicians regard the set $|R(I)|$ of all binary
relations on $I$, that is, all subsets $\rho\subseteq I\times I$,
as the underlying set of
an algebra $R(I)$ with three binary operations, two unary
operations, and three zeroary operations (constants)
\cite{mck}~\cite{wiki}.
(J\'{o}nsson \cite{BJ} calls the clone of operations that these
generate the ``classical clone'', and discusses further operations;
more on that in Section~\ref{SS.BJ}.)
The three binary operations are setwise intersection and union,
which we shall write $\cap$ and $\cup$ (also so denoted
in~\cite{BJ}, but written $\cdot$ and $+$ in~\cite{mck},
and $\wedge$ and $\vee$ in~\cite{wiki}), and composition,
which we shall write~$\circ$:
\begin{equation}\begin{minipage}[c]{23pc}\label{d.comp}
$\rho\circ\sigma\ =\ \{(i,k)\in I\times I:(\exists~j\in I)
\,\ (i,j)\in\rho,\ (j,k)\in\sigma\,\}$.
\end{minipage}\end{equation}
(This is written $\rho\,|\,\sigma$ in~\cite{BJ} and~\cite{mck}, and
$\rho\cdot \sigma$ in~\cite{wiki}.)
The two unary operations are complementation, which we may
write $^c \rho$ (but which will not come into our results;
it is denoted ``$-$'' in~\cite{BJ}, \cite{mck} and~\cite{wiki}),
and the \emph{converse} operation,
\begin{equation}\begin{minipage}[c]{23pc}\label{d.conv}
$\rho^{-1}\ =\ \{(i,j)\in I\times I:(j,i)\in\rho\,\}$
\end{minipage}\end{equation}
(denoted $\rho^{\smile}$ in~\cite{BJ}, \cite{mck} and~\cite{wiki}).
The three zeroary operations are the empty relation $0$
(set-theoretically, $\emptyset$), the total relation $1$
(set-theoretically, $I\times I$),
and the \emph{diagonal} or \emph{identity} relation,
\begin{equation}\begin{minipage}[c]{23pc}\label{d.Delta}
$\Delta\ =\ \{(i,i):i\in I\}$
\end{minipage}\end{equation}
(denoted $E$ in~\cite{BJ}, $I$ in~\cite{mck},
$\mathbf{I}$ in~\cite{wiki}).

Note that $\cap$, $\cup$, $^c$, $0$, and $1$ are the Boolean
operations on subsets of $I\times I$; so $R(I)$ is
a Boolean algebra with three additional operations,
$\circ$, $^{-1}$, and $\Delta$.

As indicated in the preceding section, we are only
interested here in \emph{reflexive} binary relations, i.e.,
relations containing $\Delta$, so the operations
$0$ and $^c$ will not concern us.
Moreover, the construction we are interested in does
not respect unions of relations: given sets $X_i$ $(i\in I)$,
the set of permutations of $X=\bigcup_{i\in I} X_i$
that allow movement of (finitely many) elements
from $X_i$ to $X_j$ whenever $(i,j)\in \rho\cup\sigma$
may be larger than the union of the set
that allows such movement only when $(i,j)\in\rho$, and
the set that allows such movement only when $(i,j)\in\sigma$.
Nor does our construction respect the diagonal relation $\Delta$ on $I$,
since $\Delta$ induces, not the trivial (``identity'') subgroup of $G$,
but the subgroup of
those permutations that carry each $X_i$ into itself.

Hence we will work with the following more restricted structures.

\begin{definition}\label{D.r(I)}
For $I$ a set, the \emph{restricted} relation algebra on $I$
will mean the algebra
\begin{equation}\begin{minipage}[c]{23pc}\label{d.r(I)}
$r(I)\ =\ (|r(I)|,\ \cap,\ \circ,\ {}^{-1},\ 1)$,
\end{minipage}\end{equation}
where $|r(I)|$ is the set of all reflexive binary relations
on $I$, and the four operations are defined as in the
above discussion of the full relation algebra $R(I)$.
\end{definition}

We now turn to groups.
For $G$ a group, the set of all subsets of $|G|$,
furnished with certain operations, is called in~\cite{mck} the
\emph{complex algebra} of $G$, apparently based on terminology
in which a subset of a group was called a complex.
Nowadays, the concepts of chain and cochain complex
are the immediate associations of that word in algebra,
so it seems best to introduce a different terminology.
We again use the word ``restricted'' to signal the limitation in
the sets and operations we allow.

\begin{definition}\label{D.p(G)}
For $G$ a group \textup{(}with identity element
written $e$\textup{)}, the \emph{restricted subset algebra}
of $G$ will mean the algebra
\begin{equation}\begin{minipage}[c]{23pc}\label{d.p(G)}
$p(G)\ =\ \{\,|p(G)|,\ \cap,\ \cdot,\ ^{-1},\ |G|\,\}$,
\end{minipage}\end{equation}
where
\begin{equation}\begin{minipage}[c]{23pc}\label{d.p|G|}
$|p(G)|\ =\ \{S\subseteq |G|:e\in X\}$,
\end{minipage}\end{equation}
and where for $S$, $T\in|p(G)|$ we define
\begin{equation}\begin{minipage}[c]{23pc}\label{d.cir}
$S\cdot T\ =\ \{gh:g\in S,\ h\in T\}$,
\end{minipage}\end{equation}
and
\begin{equation}\begin{minipage}[c]{23pc}\label{d.-1}
$S^{-1}\ =\ \{g^{-1}:g\in S\}$,
\end{minipage}\end{equation}
while letting
$\cap$ and $|G|$ in~\eqref{d.p(G)} have their obvious meanings
\textup{(}intersection of subsets, and the improper subset
of~$|G|$\textup{)}.
\end{definition}

We can now formulate
precisely the construction sketched in the preceding section.

\begin{definition}\label{D.G_(X)}
Let $(X_i)_{i\in I}$ be a family of pairwise disjoint infinite
sets, with union $X=\bigcup_I X_i$, and for each $x\in X$,
let $\iota(x)\in I$ be the index such that $x\in X_{\iota(x)}$.
We shall denote by $G_{(X_i)}$ the group of
those permutations $g$ of $X$ having the property
that $\iota(xg)=\iota(x)$ for all but finitely many $x\in X$.

For each $\rho\in|r(I)|$ \textup{(}see
Definition~\ref{D.r(I)}\textup{)}, we shall write
\begin{equation}\begin{minipage}[c]{23pc}\label{d.S_rho}
$S_\rho\ =
\ \{g\in|G_{(X_i)}|:(\forall\,x\in X)\ (\iota(x),\iota(xg))\in\rho\}$,
\end{minipage}\end{equation}
in other words, the set of those permutations of $X$
that change the home set $X_i$ of only finitely many elements
of $X$, and that can only move an
element from $X_i$ to $X_j$ if $(i,j)\in\rho$.
\end{definition}

\section{Embeddings}\label{S.main}

The first part of the next theorem is our generalization of the
asserted properties of the construction of Section~\ref{S.u_d};
we shall recover that particular case from it in the next section.
The second part of the theorem is a case of a known result
in the opposite direction.

\begin{theorem}\label{T.main}
Let $I$ be a set.
Then there exists a group $G$ and an embedding
of the restricted relation algebra $r(I)$ in the restricted
subset algebra $p(G)$.
Namely, starting with any $I$-tuple of pairwise disjoint
infinite sets $(X_i)_{i\in I}$, if we take for $G$ the group
$G_{(X_i)}$ of Definition~\ref{D.G_(X)}, then the map
$|r(I)|\to |p(G)|$ defined by
\begin{equation}\begin{minipage}[c]{23pc}\label{d.rI_to_sG}
$\rho\ \mapsto\ S_\rho\ \subseteq\ |G|$\quad
\textup{(}see~\eqref{d.S_rho}\textup{)}
\end{minipage}\end{equation}
is such an embedding.

Inversely, given any group $G$, there exists a set $I$
and an embedding of $p(G)$ in $r(I)$.
Namely, for $I=|G|$, such an embedding is given by
\begin{equation}\begin{minipage}[c]{23pc}\label{d.sG_to_rI}
$S\ \mapsto\ \{(g,gs):g\in|G|,\ s\in S\}$,
\end{minipage}\end{equation}
equivalently, $\{(g,h)\in |G|\times|G|:g^{-1} h\in S\}$.

Thus, if $A$ is any algebra with four operations, two binary, one unary
and one zeroary, then $A$ is embeddable in the restricted
relation algebra $r(I)$ of some set $I$ if and only if it is
embeddable in the restricted subset algebra $p(G)$ of some group $G$.
\end{theorem}

\begin{proof}
It is immediate from the definition~\eqref{d.S_rho}
that the map~\eqref{d.rI_to_sG} respects intersections, takes
converse relations to inverse subsets, and takes the improper
relation on $I$ to the whole group $G_{(X_i)}$; so what we need to
verify is that it respects composition and is one-to-one.
It is straightforward that
$S_\rho\cdot S_\sigma\subseteq S_{\rho\circ\sigma}$.
Indeed, writing an element of $S_\rho\cdot S_\sigma$
as $gh$ with $g\in S_\rho$ and $h\in S_\sigma$, consider
any $x\in X$.
By~\eqref{d.S_rho} $(\iota(x),\iota(xg))\in\rho$ and
$(\iota(xg),\iota(xgh))\in\sigma$;
hence $(\iota(x),\iota(xgh))\in\rho\circ\sigma$.
Since this holds for all $x\in X$, we have
$gh\in S_{\rho\circ\sigma}$, giving the asserted inclusion.

To get the reverse inclusion,
consider any $f\in S_{\rho\circ\sigma}$, and let us try to write
it as the product of a member of $S_\rho$ and a member of $S_\sigma$.
To this end, let $x_1,\dots,x_n$ be the finitely many elements
$x\in X$ such that $\iota(xf)\neq\iota(x)$
(see Definition~\ref{D.G_(X)}).
Since $f\in S_{\rho\circ\sigma}$, each pair
$(\iota(x_m),\,\iota(x_m f))$
lies in $\rho\circ\sigma$, so we can
find $j_1,\dots,j_n\in I$ such that
\begin{equation}\begin{minipage}[c]{23pc}\label{d.via_jm}
$(\iota(x_m),\,j_m)\in\rho,\quad (j_m,\,\iota(x_m f))\in\sigma
\quad (m=1,\dots,n)$.
\end{minipage}\end{equation}
For $m=1,\dots,n$, let us choose elements $y_m\in X_{j_m}$
so that $y_1,\dots,y_n$ are distinct
(which is possible because the $X_i$ are all infinite).
Let us now choose any $g\in|G|$ such that
\begin{equation}\begin{minipage}[c]{23pc}\label{d.xmg}
$x_m\,g\ =\ y_m$ for $m=1,\dots,n$,
while $g$ moves no other elements of $X$ out of the sets $X_i$ in
which they started.
\end{minipage}\end{equation}
This is again possible because the $X_i$ are infinite, so that
if the number of elements to be moved into and out of a
given $X_i$ are different, we can absorb the disparity by moving
infinitely many elements within $X_i$.
By~\eqref{d.xmg} and
the first relation of~\eqref{d.via_jm}, $g\in S_\rho$.
Letting $h=g^{-1} f$, we see that no element not among the $y_m$
is moved by $h$ from one $X_i$ to another, while each $y_m$
is moved from $X_{j_m}$ to $X_{\iota(x_m f)}$; so
by the second condition of~\eqref{d.via_jm}, $h\in S_\sigma$.
Since $f=gh$, this shows that the arbitrary element
$f\in S_{\rho\circ\sigma}$ lies in $S_\rho\cdot S_\sigma$,
completing the proof that~\eqref{d.rI_to_sG} respects composition.

To show that~\eqref{d.rI_to_sG} is one-to-one,
suppose $\rho$ and $\sigma$ are distinct elements of $|r(I)|$.
Assume without loss of generality that $(i,j)$
belongs to $\rho$ but not to $\sigma$.
Then $i\neq j$, and
we can construct a permutation $g$ of $X$ which moves one
element of $X_i$ into $X_j$, and keeps all other elements
of $X$ in their home sets.
Thus $g$ belongs to $S_\rho$ but not to $S_\sigma$,
showing that these are distinct, and completing the proof of
the first assertion of the theorem.

That the construction~\eqref{d.sG_to_rI} has the properties
stated in the second assertion of the theorem is immediate,
as is the final consequence of this pair of assertions.
\end{proof}

Remarks:  The homomorphism~\eqref{d.sG_to_rI} is a standard way
of embedding the (unrestricted) ``complex algebra'' of a group in a
relation algebra (also in the unrestricted sense).
But the behavior of the construction~\eqref{d.rI_to_sG} proved above
is quite different from the situation
for unrestricted relation algebras:
R.\,McKenzie~\cite{mck} shows, roughly, that the class of subalgebras
of unrestricted relation algebras on sets that are
embeddable (via \emph{any} map) in subset algebras
of groups is not determined by any finite set of first-order axioms.

The fact that the homomorphisms~\eqref{d.rI_to_sG}
and~\eqref{d.sG_to_rI} are one-to-one and respect intersections
shows that they are embeddings of posets.
It also is not hard to see that in addition to the finitary
operations of Theorem~\ref{T.main}, they respect
unions and intersections of infinite chains.

\section{Applications}\label{S.appl}

\subsection{Back to where we began.}\label{SS.u_d}
Let us recover from Theorem~\ref{T.main} the properties asserted
in Section~\ref{S.u_d} for the group and monoid described there.
We will need the following observation.

\begin{lemma}\label{L.gp,mnd}
For $(X_i)_{i\in I}$ as in Definition~\ref{D.G_(X)},
and any $\rho\in|r(I)|$, the set $S_\rho$ is a submonoid
of $G_{(X_i)}$ if and only if $\rho$ is a preorder, while it is a
subgroup if and only if $\rho$ is an equivalence
relation.
\end{lemma}

\begin{proof}
Since any $\rho\in|r(I)|$ is reflexive, $\rho$ will be a preorder
if and only if it is transitive, i.e., satisfies
$\rho\circ\rho\subseteq\rho$, which by Theorem~\ref{T.main}
is equivalent to the condition that
$S_\rho\cdot S_\rho \subseteq S_\rho$.
(Recall the remark at the end of the preceding section, that
the constructions of that theorem are embeddings of posets.)
Since $S_\rho$ contains the identity element
of $G_{(X_i)}$, the above is
precisely the condition for $S_\rho$ to be a submonoid thereof.

A preorder on a set is an equivalence relation if and only it
is symmetric, i.e., if and only if
$\rho=\rho^{-1}$, which by Theorem~\ref{T.main}
is equivalent to $S_\rho=S_\rho^{-1}$, which says that the monoid
$S_\rho$ is a subgroup of $G_{(X_i)}$.
\end{proof}

In particular, for $I=\{1,\dots,n\}$ with the partial
ordering $\preccurlyeq$ described at the start of Section~\ref{S.abstr},
$S_\preccurlyeq$ is a submonoid $M$ of $G=G_{(X_i)}$.
If we form composites
$\preccurlyeq\circ\preccurlyeq^{-1}\circ\preccurlyeq\circ\dots$,
it is easy to verify that the length-$n$ composite is the indiscrete
relation $I\times I$, but that no initial subcomposite is.
(Specifically, one sees by induction
that in building up that composite,
we do not get the pair $(1,m)$ until we have $m$ factors; so
in particular, we don't get $(1,n)$ till we have all $n$ factors.)
Hence, as claimed, the $n$-fold product $M\cdot M^{-1}\cdot\dots$
equals $G$, but no initial subproduct does.

\subsection{Does the left hand know what the right hand is doing?}\label{SS.u_d_vs_d_u}
Suppose we start with an arbitrary finite partially
ordered set $(I,\preccurlyeq)$ and a family $(X_i)_{i\in I}$
of pairwise disjoint infinite sets.
Because $I$ is finite, the increasing chain of binary
relations on $I$,
\begin{equation}\begin{minipage}[c]{23pc}\label{d.prec...}
$\preccurlyeq,\quad\preccurlyeq\circ\preccurlyeq^{-1},\quad
\preccurlyeq\circ\preccurlyeq^{-1}\circ\preccurlyeq,\quad\dots$
\end{minipage}\end{equation}
will stabilize after finitely many steps;
hence so will the chain of subsets
$S_\preccurlyeq\cdot S_\preccurlyeq^{-1}\cdot\dots$ of $G=G_{(X_i)}$.
Clearly, the eventual value of the
latter chain will be the subgroup generated
by $S_\preccurlyeq$ (corresponding to the equivalence relation
on $I$ generated by $\preccurlyeq$).

What if, instead of starting our products with $S_\preccurlyeq$,
we begin with $S_\preccurlyeq^{-1}$?
The chain of subsets
$S_\preccurlyeq^{-1}\cdot S_\preccurlyeq\cdot\dots$ that we get
will necessarily stabilize at the same group; but will the two
chains achieve grouphood after the same numbers of steps?

If one of them achieves grouphood with an odd number of
factors, then the other will do so at or before the same step; for the
odd-length products have the forms
$S_\preccurlyeq\cdot S_\preccurlyeq^{-1}\cdot\,\ldots
\,\cdot S_\preccurlyeq$ and
$S_\preccurlyeq^{-1}\cdot S_\preccurlyeq\cdot\,\ldots
\,\cdot S_\preccurlyeq^{-1}$,
which are inverses to one another; so if one is a group,
the other is the same group.
However, it is possible for one of the products,
$S_\preccurlyeq\cdot S_\preccurlyeq^{-1}\cdot\,\ldots$
or $S_\preccurlyeq^{-1}\cdot S_\preccurlyeq\cdot\,\ldots$,
to achieve grouphood at an even length $2m$,
while the other does not do so till it
reaches length $2m+1$ (at which stage it indeed gives the
same group, since its \emph{last} $2m$ factors form the product
we have assumed is a group).
In fact, for the poset of~\eqref{d.u_d} with
$n=2m+1$, we find that the $n{-}1$-fold product beginning
with $S_\preccurlyeq^{-1}$ equals $G$, while we have seen that
the product starting with
$S_\preccurlyeq$ needs $n$ steps to get to $G$.

There are also partially ordered sets for which the two sorts of
composites, $S_\preccurlyeq\cdot S_\preccurlyeq^{-1}\cdot\dots$ and
$S_\preccurlyeq^{-1}\cdot S_\preccurlyeq\cdot\dots$, reach grouphood at
the same even or odd step.
For even $n$, this is true of the construction of Section~\ref{S.u_d},
as may be deduced from the fact that in that case,
$(I,\preccurlyeq)$ and $(I,\preccurlyeq^{-1})$ are isomorphic posets.
For odd $n$, we can take $(I,\preccurlyeq)$ to be
the union of a copy of the $n$-element
poset constructed as in Section~\ref{S.u_d}, and a copy of its
opposite, with the elements of one copy
incomparable with those of the other.
A construction that works
uniformly for even and odd $n$ takes $I=|\Z/2(n-1)\Z|$,
again letting $\preccurlyeq$ be the partial order under which each even
element of $I$ is $\preccurlyeq$ the two adjacent odd elements.
(Thus, $(I,\preccurlyeq)$ is a ``crown'' with $2(n-1)$ vertices.)

\subsection{Another group-from-monoids example.}\label{SS.135}
Let $I=\{1,2,3,4,5,6\}$, and consider the following
three partial orderings on~$I$.
\begin{equation}\begin{minipage}[c]{23pc}\label{d.135}
\hspace*{-.17em}$\preccurlyeq_1$, under which
$1\preccurlyeq_1 2\preccurlyeq_1 3\preccurlyeq_1 4
\preccurlyeq_1 5$, while $6$ is incomparable with all of these,\\[.4em]
$\preccurlyeq_3$, under which
$3\preccurlyeq_3 4 \preccurlyeq_3 5
\preccurlyeq_3 6\preccurlyeq_3 1$,
while $2$ is incomparable with all of these,\\[.4em]
$\preccurlyeq_5$, under which
$5 \preccurlyeq_5 6\preccurlyeq_5
1\preccurlyeq_5 2 \preccurlyeq_5 3$,
while $4$ is incomparable with all of these.
\end{minipage}\end{equation}
I claim that
\begin{equation}\begin{minipage}[c]{23pc}\label{d.135=1}
$(\preccurlyeq_5\circ\preccurlyeq_3\circ\preccurlyeq_1)\ =
\ (\preccurlyeq_3\circ\preccurlyeq_1\circ\preccurlyeq_5)\ =
\ (\preccurlyeq_1\circ\preccurlyeq_5\circ\preccurlyeq_3)\ =\ 1$,\\[.4em]
but none of
$(\preccurlyeq_1\circ\preccurlyeq_3\circ\preccurlyeq_5)$,
$(\preccurlyeq_5\circ\preccurlyeq_1\circ\preccurlyeq_3)$, or
$(\preccurlyeq_3\circ\preccurlyeq_5\circ\preccurlyeq_1)$ equals~$1$.
\end{minipage}\end{equation}
(Above and in the next paragraph, parentheses are put around
relation-symbols when set-theoretic relation symbols are
applied to them.)
From this, it will follow by Theorem~\ref{T.main}
that the submonoids $A= S_{\preccurlyeq_5}$,
$B=S_{\preccurlyeq_3}$ and $C=S_{\preccurlyeq_1}$ of $G$ satisfy
\begin{equation}\begin{minipage}[c]{23pc}\label{d.ABC}
$A\cdot B\cdot C\ =\ B\cdot C\cdot A\ =\ C\cdot A\cdot B\ =\ G$,\\[.2em]
but none of $ A\cdot C\cdot B$ or $B\cdot A\cdot C$ or
$C\cdot B\cdot A$ equals $G$.
\end{minipage}\end{equation}

To establish~\eqref{d.135=1}, let us first verify that
$(\preccurlyeq_5\circ\preccurlyeq_3\circ\preccurlyeq_1)=1$.
Given any $i\in I$, we must show that
$(i,j)\in(\preccurlyeq_5\nolinebreak\circ
\preccurlyeq_3\circ\preccurlyeq_1)$ for all $j\in I$.
We see from the last line of~\eqref{d.135} that if $i\neq 4$, then
$(i,3)\in(\preccurlyeq_5)$, while if $i=4$ we trivially have
$(i,4)\in(\preccurlyeq_5)$.
From these two relations and the second
line of~\eqref{d.135}, we can see that whatever $i$ is,
$\preccurlyeq_5\circ\preccurlyeq_3$ contains $(i,4)$, $(i,5)$,
$(i,6)$ and $(i,1)$.
Composing with $\preccurlyeq_1$, and using the fact that
we have already gotten $(i,1)$, we see from the first line
of~\eqref{d.135} that we get $(i,2)$ and $(i,3)$ as well.
So $\preccurlyeq_5\circ\preccurlyeq_3\circ\preccurlyeq_1$
indeed contains all pairs $(i,j)$ $(i,j\in I)$.
The same is true of the other two composites listed in~\eqref{d.135=1},
by symmetry, i.e., by the fact that the $2$-step
cyclic permutation of the elements of $I$ cyclically permutes the
preorders~\eqref{d.135}.
Thus we have established the first line of~\eqref{d.135=1}.

On the other hand, it is not hard to check that
$\preccurlyeq_1\circ\preccurlyeq_3\circ\preccurlyeq_5$
does not contain $(5,4)$, $(6,4)$ or $(6,5)$,
hence is not the improper relation $1$;
and again, by symmetry, this implies the
other cases of the second line of~\eqref{d.135=1}.
As noted, these results give~\eqref{d.ABC}.

\section{Some variant constructions}\label{S.variants}

\subsection{Generalizing the finite/infinite contrast.}\label{SS.small/big}
There are some easy variants of
Definitions~\ref{D.r(I)}-\ref{D.G_(X)} for which
Theorem~\ref{T.main} goes over without difficulty.

On the one hand, one can everywhere replace ``finite'' and ``infinite''
by ``of cardinality $<\kappa$'' and ``of cardinality $\geq\kappa$''
for any fixed infinite cardinal $\kappa$.

If $I$ is infinite, one can also weaken the condition that only
finitely many (or in the above generalization, fewer than $\kappa$)
elements of $X$ move from one $X_i$ to another, to merely say
that each $X_i$ receives only finitely many (respectively, fewer
than $\kappa$) elements from outside
itself, and sends only finitely many
(respectively, fewer than $\kappa$) out of itself.

One can also replace infinite sets $X_i$, their finite subsets, and
set maps among them by, say, measure
spaces of positive measure, subsets of measure zero, and
measure-preserving permutations of $X$,
if one uses measure spaces for which one has an appropriate
measure-theoretic version of Hilbert's Hotel.
In particular, the standard measure on the real unit interval $[0,1]$
has the property that for every
$Z\subseteq [0,1]$ of measure zero, there is a measure-preserving
bijection between $[0,1]-Z$ and $[0,1]$
(Charles Pugh, personal communication).
From this it easily follows that $X=\bigcup X_i$ has the desired
properties if each $X_i$ is a copy of $[0,1]$, and $G$ consists
of the measure-preserving permutations of $X$ under which
the set of points that move between the sets $X_i$ has measure zero.

\subsection{A finitely generated example.}\label{SS.fg}
Dawid Kielak (personal communication)
has asked whether one can get
a $G$ and an $M$ answering Kuczma's original
question, such that $M$ is finitely generated as a monoid.
I outline below how to modify the construction of~\S\ref{S.u_d}
to get such an example.

Given a finite poset $(I,\preccurlyeq)$, let $X=\bigcup_{i\in I} X_i$,
where each $X_i$ consists of elements $x_{i,k}$ $(k\in\Z)$, and
elements with distinct subscript-pairs are understood to be distinct.
Let $M$ be the monoid of permutations of $X$ generated by the
following elements:
\begin{equation}\begin{minipage}[c]{23pc}\label{d.a_i}
$a_i$, defined for each $i\in I$, which acts by
$x_{i,k}\,a_i = x_{i,k+1}$, and fixes all $x_{i',k}$ with $i'\neq i$,
\end{minipage}\end{equation}
\begin{equation}\begin{minipage}[c]{23pc}\label{d.a_i-}
$a_i^{-1}$, defined for each $i\in I$, which, of course, acts
by $x_{i,k}\,a_i^{-1} = x_{i,k-1}$, and again fixes
all $x_{i',k}$ with $i'\neq i$,
\end{minipage}\end{equation}
\begin{equation}\begin{minipage}[c]{23pc}\label{d.b_i}
$b_i$, defined for each $i\in I$, which interchanges
$x_{i,0}$ and $x_{i,1}$, and fixes all other elements of $X$,
\end{minipage}\end{equation}
\begin{equation}\begin{minipage}[c]{23pc}\label{d.c_i,i'}
$c_{i,i'}$ defined for all $i\neq i'$
such that $i\preccurlyeq i'$, which acts by\\[.1em]
\hspace*{1in}$x_{i,0}\ c_{i,i'} \ = \ x_{i',0}$,\\[.1em]
\hspace*{1in}$x_{i,k}\ c_{i,i'} \ = \ x_{i,k-1}$, for $k>0$,\\[.1em]
\hspace*{1in}$x_{i',k}\ c_{i,i'} \,= \ x_{i',k+1}$,
for $k\geq 0$,\\[.1em]
and which fixes all elements of $X$ other than those listed above
\textup{(}including all elements $x_{i,k}$ and $x_{i',k}$
with $k<0$\textup{)}.
\end{minipage}\end{equation}

Note that the generators in~\eqref{d.a_i},
\eqref{d.a_i-} and~\eqref{d.b_i}
all have inverses in $M$, so the monoid they generate is a group $H$.
Clearly, members of $H$ carry each $X_i$ into itself;
it is not hard to verify that $H$
\emph{includes} all permutations with this property
which move only finitely many elements of $X$
(these form the subgroup generated by the conjugates
of the $b_i$ by powers of the $a_i$),
and that the general element of $H$
acts on each $X_i$ as such a permutation, followed by
a translation $a_i^{t_i}$ $(t_i\in\Z)$.

It is also not hard to verify that when we bring in the $c_{i,i'}$,
the effect is that whenever $i\prec i'$, an element $f\in M$ can
move any finite family of elements of $X_i$ to arbitrary
positions in $X_{i'}$.
(The idea is to move the elements of $X_i$ that are to be
transferred to $X_{i'}$ one by one to the position $x_{i,0}$,
then apply $c_{i,i'}$; and then move their images, which are
at first $x_{i',0}$, to the desired positions in $X_{i'}$.)
The actions of $f$ on ``most'' elements of each $X_i$ will still be
by translations of the second subscript; but a consequence of the
difference between the behavior the $c_{i,i'}$ on elements with
positive and with negative second subscripts is that for each $i$,
there will, in general, be two different translations, one affecting
elements $x_{i,k}$ with $k$ large (above some constant),
the other affecting such elements with $k$ small (below some constant).
Precisely,
\begin{equation}\begin{minipage}[c]{23pc}\label{d.M}
A permutation $f$ of $X$ will belong to $M$ if and only if
(i)~for all $i,i'\in I$, $f$ carries no elements of $X_i$
into $X_{i'}$ unless $i\preccurlyeq i'$, and
(ii)~for each $i\in I$, there exist an integer $s_i$ such that
for all sufficiently large $k$, $x_{i,k}f = x_{i,k+s_i}$, and
an integer $t_i$ such that
for all sufficiently small $k$, $x_{i,k}f = x_{i,k+t_i}$.
\end{minipage}\end{equation}
Here $s_i-t_i$ will be the number of elements that $f$ moves into
$X_i$, minus the number it moves out of $X_i$.
It is not hard to deduce from~\eqref{d.M}, by the same method
used in the proof of Theorem~\ref{T.main} to show
that $S_\rho\,S_\sigma = S_{\rho\circ\sigma}$,
that membership in a product
$M\,M^{-1}M\,\dots\,M^{(-1)^{m-1}}$
is described by the condition like~\eqref{d.M},
except that in~(i), the relation $\preccurlyeq$ is replaced by
$\preccurlyeq\circ\preccurlyeq^{-1}\circ\preccurlyeq\circ
\dots\circ\preccurlyeq^{(-1)^{m-1}}$.
In particular, letting $I$ be the poset of~\eqref{d.u_d},
we see that $M\,M^{-1} M\dots\,M^{(-1)^{m-1}}$ will be
the group $G$ of all permutations that satisfy~(ii)
if and only if $m\geq n$; from which it follows that $M$
has the property asked for by Kuczma.

We remark that if $I$ is any finite set, and for
each $\rho\in |r(I)|$ we let $S_\rho$ be the
set of permutations $f$ of $X$ satisfying the analog of~\eqref{d.M}
with $\rho$ in place of $\preccurlyeq$ in~(i), then this
gives us an embedding of $r(I)$ in $p(G)$, as in Theorem~\ref{T.main}.
But if $\rho$ is not a partial order (or more generally, a preorder),
I do not see any property of $S_\rho$ generalizing
the striking fact that when it is a preorder,
$M$ is finitely generated as a monoid.
(The closest I can see is the observation that for all $\rho\in|r(I)|$,
the submonoid of $G$ generated by $S_\rho$ is finitely generated.)

\subsection{Dropping the finiteness restriction.}\label{SS.drop_fin}
We might enlarge, rather than restricting, the class of
the permutations we consider.
Given a finite set $I$ and
a family of disjoint infinite sets $X_i$, say,
for simplicity, all countable,
suppose we let $G$ be the group of \emph{all} permutations of
$X=\bigcup X_i$, without any finiteness restriction on
elements moving among the $X_i$.
Let us again associate
to each binary relation $\rho$ on $I$
the subset of those elements of $G$ that don't
move elements from $X_i$ to $X_j$ unless $(i,j)\in\rho$,
and call this $S_\rho$.

Again, if $\preccurlyeq$ is a preorder, then
$S_\preccurlyeq$ is a monoid; and we can again get examples where the
group generated by $S_\preccurlyeq$ can be expressed as a finite
product of $S_\preccurlyeq$ and $S_\preccurlyeq^{-1}$,
but where the smallest such product is arbitrarily long.
But the monoid $M$ determined by a given pair
$(I,\preccurlyeq)$ under this version of our construction
may in fact require longer products
$M\cdot M^{-1}\cdot\dots$ to get all of $G$ than
the monoid described in Sections~\ref{S.defs}-\ref{S.main} took
to generate the $G$ of that construction.
For instance, if $I=\{1,2\}$ with $1\preccurlyeq 2$, I claim
that a composite
of $S_\preccurlyeq$ and $S_\preccurlyeq^{-1}$ which contains a
permutation that \emph{interchanges} the contents of
the two sets $X_1$ and $X_2$ must have length at least~$3$.
To see that no such permutation belongs to
$S_\preccurlyeq\cdot S_\preccurlyeq^{-1}$,
note that any member of $S_{\preccurlyeq}$
must leave in $X_1$ infinitely
many of the elements that were originally there,
and that none of these is moved out by any member of
$S_{\preccurlyeq}^{-1}$.
But we do get permutations interchanging $X_1$ and $X_2$ in
$S_\preccurlyeq\cdot S_\preccurlyeq^{-1}\cdot S_\preccurlyeq$:
think (roughly speaking)
of first moving half the elements of $X_1$ up into $X_2$
(as in the versions of Hilbert's Hotel with infinitely many
arriving or departing guests), then moving
all the elements that were originally in $X_2$ down into $X_1$,
and finally moving those that stayed in $X_1$ at the
first step up into $X_2$.

In this situation, the class
of sets $S_\rho\subseteq G$ $(\rho\in|r(I)|)$
is not closed under composition; e.g., it is not hard to
see that in the situation just described,
$S_\preccurlyeq\cdot S_\preccurlyeq^{-1}
=S_\preccurlyeq\cdot S_{\preccurlyeq^{-1}}$
does not have the form $S_\rho$ for any $\rho\in|R(I)|$.
I have not examined for
general $I$ the algebra of subsets of $|G|$ generated by
the sets $S_\rho$ $(\rho\in|R(I)|)$
under the operations of $p(G)$.

\subsection{A two-group construction.}\label{SS.G,H}
In the discussion preceding
Definitions \mbox{\ref{D.r(I)}-\ref{D.G_(X)}},
we noted that the construction we were
leading up to would not respect ``identity elements'',
i.e., would not send $\Delta\in|r(I)|$ to $\{e\}\in|p(G)|$; so we
left the zeroary operations $\Delta$ and $\{e\}$ out of the structures
we defined.
However, with a bit of added complication,
we can bring $\Delta$ back in.
Note that in the context of Definition~\ref{D.G_(X)},
$S_\Delta$ is the subgroup $H\subseteq G$ of permutations
of $X$ that carry each $X_i$ to itself; and that every
subset $S_\rho$ is
left and right $H$-invariant (closed under left and right
multiplication by elements of $H$).
Building on this observation, suppose that
for any group $G$ and subgroup $H$, we let
\begin{equation}\begin{minipage}[c]{23pc}\label{d.|p'|}
$|p'(G,H)|\ =\ \{S\subseteq|G|:e\in S,\ S=HSH\}$.
\end{minipage}\end{equation}

This family of sets is closed under the four operations in our
definition of $p(G)$, and also contains the set $|H|$.
We find that if we now define algebras with five operations,
\begin{equation}\begin{minipage}[c]{23pc}\label{d.p'}
$p'(G,H)\ =\ (|p'(G,H)|,\ \cap,\ \cdot,\ ^{-1},\ |H|,\ |G|)$,
\end{minipage}\end{equation}
\begin{equation}\begin{minipage}[c]{23pc}\label{d.r'}
$r'(I)\ =\ (|r(I)|,\ \cap,\ \circ,\ \ ^{-1},\ \Delta,\ 1)$,
\end{minipage}\end{equation}
and for any family of disjoint infinite sets $(X_i)_{i\in I}$, we
supplement Definition~\ref{D.G_(X)} with
\begin{equation}\begin{minipage}[c]{23pc}\label{d.H}
$H_{(X_i)}\ =$ the subgroup of $G_{(X_i)}$ with underlying
set $S_\Delta$,
\end{minipage}\end{equation}
then the map~\eqref{d.rI_to_sG}, i.e.,
$\rho\mapsto S_\rho$, gives an embedding
of $r'(I)$ in $p'(G_{(X_i)},\,H_{(X_i)})$
carrying the constant $\Delta$ to the constant $|H|$,
as well as respecting the other four operations.
Inversely, given a group $G$ and a subgroup $H$, if
we let $I=G/H$, the set of left cosets of $H$ in $G$,
and take as our analog of~\eqref{d.sG_to_rI}
the map $|p'(G,H)|\to |r'(I)|$ given by
\begin{equation}\begin{minipage}[c]{23pc}\label{d.s'G_to_r'I}
$S\ \mapsto\ \{(gH,\,gsH):g\in|G|,\ s\in S\}$,
\end{minipage}\end{equation}
then together these constructions
satisfy the analog of Theorem~\ref{T.main}.

(In a preprint version of this note, I claimed incorrectly
that we could get such a construction using for $I$
the set $H\backslash G/H$ of double cosets of $H$ in $G$.
The analog of~\eqref{d.s'G_to_r'I} would have been
\begin{equation}\begin{minipage}[c]{23pc}\label{d.s''G_to_r''I}
$S\ \mapsto\ \{(HgH,\,HgsH):g\in|G|,\,s\in S\}$.
\end{minipage}\end{equation}
However, this map does not respect intersections.
For instance, if $G$ is the free group on $\{x,y,z\}$ and
$H$ its subgroup generated by $x$, consider the elements
$S=H\cup HyH$ and $T=H\cup Hz^{-1}xzyH$ of $|p'(G,H)|$.
I claim that the images in $I\times I$ under~\eqref{d.s''G_to_r''I}
of both $S$ and $T$ contain the pair $(HzH,\,HzyH)$.
That the image of $S$ does is clear;
that the image of $T$ also does can be seen by
writing $(HzH,\,HzyH)$ as $(HzH,\,Hz\,{\cdot}\,z^{-1}xzyH)$.
However, the image under~\eqref{d.s''G_to_r''I} of
$S\cap T=H$ clearly does not contain $(HzH,\,HzyH)$.)

\subsection{Still more operations.}\label{SS.BJ}
In \cite[Section~1.1]{BJ} B.\,J\'{o}nsson notes that there
are many more natural operations on the binary relations
on a set $I$ than those in the ``classical clone'', the clone
generated by the operations here denoted
$\cap$, $\cup$, $\circ$, $^c$, $^{-1}$, $0$, $\Delta$ and $1$.

For example, given $\rho,\,\sigma\in|R(I)|$, he defines the
left and right \emph{residuals} of $\rho$ with respect to $\sigma$:
\begin{equation}\begin{minipage}[c]{23pc}\label{d.resid}
$\rho/\sigma\ =\ \{(i,j):
(\forall\ k\in I)\ (j,k)\in\sigma\implies(i,k)\in\rho\},
\\[.4em]
\sigma\backslash\,\rho\ =\ \{(i,j):
(\forall\ k\in I)\ (k,i)\in\sigma\implies(k,j)\in\rho\}$.
\end{minipage}\end{equation}
These are the largest relations $\tau_1$ and $\tau_2$ such
that $\tau_1\circ\sigma\subseteq\rho$ and
$\sigma\circ\tau_2\subseteq\rho$.
Unfortunately, these two operations cannot be fitted into our
restricted relation algebras, because they do
not, in general, take reflexive relations to reflexive relations.
Indeed, if $\rho/\sigma$, respectively, $\sigma\backslash\rho$, is
reflexive, this implies, by the characterization of
those operations just mentioned, that
$\Delta\circ\sigma\subseteq\rho$, respectively
$\sigma\circ\Delta\subseteq\rho$, which is not the case
for all reflexive $\rho$ and $\sigma$,
since $\Delta\circ\sigma=\sigma=\sigma\circ\Delta$.
Conceivably, it might be useful to
bring in $/$ and $\backslash$ as \emph{partial} operations,
defined on those pairs $(\rho,\sigma)$ such
that $\sigma\subseteq\rho$; or, roughly equivalently, to look at the
operations $\rho/(\sigma\cap\rho)$ and $(\sigma\cap\rho)\backslash\rho$.

However, another
(infinite) family of operations discussed in \cite[Section~1.1]{BJ}
can be incorporated nicely into our restricted relation algebras.
These are typified by the $5$-ary operation $Q$ defined by
\begin{equation}\begin{minipage}[c]{23pc}\label{d.Q}
$Q(\rho_1,\,\sigma_1,\,\rho_2,\,\sigma_2,\,\tau)\ =
\{(i,j)\in I\times I:(\exists\ k,\,\ell\in I)\\[.2em]
\hspace*{3em}(i,k)\in\rho_1,\ (k,j)\in\sigma_1,
\ (i,\ell)\in\rho_2,\ (\ell,j)\in\sigma_2,\ (k,\ell)\in\tau\}$
\end{minipage}\end{equation}
\cite[p.\,247]{BJ}.
One finds that this corresponds as in
Theorem~\ref{T.main} to the $5$-ary operation on
subsets of a group $G$,
\begin{equation}\begin{minipage}[c]{23pc}\label{d.Q_gp}
$Q(R_1,\,S_1,\,R_2,\,S_2,\,T)\ =
\ \{f\in|G|:(\exists\ g, h\in |G|)\\[.2em]
\hspace*{3em}g\in R_1,\ g^{-1}f\in S_1,
\ h\in R_2,\ h^{-1}f\in S_2,\ g^{-1}h\in T\}\,$.
\end{minipage}\end{equation}
Each of the operations in the family exemplified by $Q$ is
determined by a finite directed graph with two distinguished vertices,
and edges labeled by the arguments of the operation.
E.g., $Q$ corresponds to the graph
\raisebox{2.5pt}[12pt][7pt]{
\begin{picture}(34,12)
\put(2,0){\circle*{2}}
\put(4,1){\vector(2,1){10}}
\put(2,0){\circle*{2}}
\put(4,-2){\vector(2,-1){10}}
\put(16,7){\circle*{2}}
\put(18,6){\vector(2,-1){10}}
\put(16,-7){\circle*{2}}
\put(18,-7){\vector(2,1){10}}
\put(16,5){\vector(0,-1){10}}
\put(30,0){\circle*{2}}
\end{picture}},
where the distinguished
vertices are those at the left and right ends, and
where $\rho_1$ and $\sigma_1$ label the top two edges,
$\rho_2$ and $\sigma_2$ the bottom two,
and $\tau$ the vertical edge~\cite[p.\,248]{BJ}.
(The operations $\cap$, $\circ$, $^{-1}$,
$\Delta$ and $1$ also correspond in this way to graphs, namely
\raisebox{2.5pt}{
\begin{picture}(20,6)
\put(2,0){\circle*{2}}
\put(5,2){\vector(1,0){10}}
\put(5,-2){\vector(1,0){10}}
\put(18,0){\circle*{2}}
\end{picture}},
\raisebox{2.5pt}{
\begin{picture}(36,0)
\put(2,0){\circle*{2}}
\put(5,0){\vector(1,0){10}}
\put(18,0){\circle*{2}}
\put(21,0){\vector(1,0){10}}
\put(34,0){\circle*{2}}
\end{picture}},
\raisebox{2.5pt}{
\begin{picture}(22,3)
\put(2,0){\circle*{2}}
\put(15,0){\vector(-1,0){10}}
\put(18,0){\circle*{2}}
\end{picture}},
\raisebox{2.5pt}{
\begin{picture}(5,3)
\put(2,0){\circle*{2}}
\end{picture}},\hspace*{.3em}
and
\raisebox{2.5pt}{
\begin{picture}(12,3)
\put(2,0){\circle*{2}}
\put(9,0){\circle*{2}}
\end{picture}},\hspace*{.3em}
in each case with the leftmost and rightmost vertices distinguished,
and with appropriate labeling of the edges by
the arguments of the operation.
On the other hand, the operations $\cup$, $^c$, and $0$
do not belong to this family, nor do
the two operations of~\eqref{d.resid}.)

For each such labeled graph, we likewise get an operation on subsets
of groups which relates to the operation on binary relations
as in Theorem~\ref{T.main}.

\subsection{Some thoughts and questions}\label{SS.So?}
I do not know whether the variants
sketched in Sections~\ref{SS.small/big}-\ref{SS.BJ}
of the main construction of this paper are likely to prove
``useful'', either in answering group-theoretic
questions not answered by Theorem~\ref{T.main},
or in other ways (e.g., whether the measure-theoretic variant of
our construction might give some information on the structure
of measurable maps among measure spaces).

The referee has raised the question of which finite relation
algebras can be embedded in the restricted subset algebras $p(G)$
of \emph{finite} groups $G$.
One property that a relation algebra which can be
so embedded must have is that every
$\circ$-\emph{idempotent} element
is $^{-1}$-{\nolinebreak}invariant,
since in $p(G)$, every element idempotent with
respect to ``$\cdot$''
is a nonempty subset of $G$ closed under multiplication,
hence, if $G$ is finite, a subgroup.
I don't know whether there are additional restrictions.

The referee has also asked whether Theorem~\ref{T.main}
can be generalized to non-reflexive relations.
Since our ``Hilbert's Hotel'' trick is based on keeping
``most'' members of each $X_i$ within $X_i$, it has no
obvious extension to the non-reflexive case; so this interesting
question would require a different approach.

\subsection{A family resemblance.}\label{SS.near_u_d}
We end with an observation of a different sort.

I claim that every case of
a group $G$, a submonoid $M$ of $G$, and
a positive integer $n$ such that
\begin{equation}\begin{minipage}[c]{23pc}\label{d.not_till_n}
the $n$-fold
product $M\,M^{-1} M\,M^{-1} \dots\,M^{(-1)^{n-1}}$ equals $G$,\\
but the product of the first $n-1$ of these factors does not,
\end{minipage}\end{equation}
closely resembles, in a way, the example of Section~\ref{S.u_d}.
To see this, let $X=G$, on which we let $G$ act
by right translation, and let $X_1=M$,
$X_2=M\,M^{-1} - M$ (relative complement),
$X_3=M\,M^{-1} M - M\,M^{-1}$, and so on.
By~\eqref{d.not_till_n}, $X_n\neq\emptyset$ but $X_{n+1}=\emptyset$,
from which it is easy to deduce that
$X_i$ is nonempty for $1\leq i\leq n$, and empty for all $i>n$.

I claim that
\begin{equation}\begin{minipage}[c]{23pc}\label{d.odd}
$X_i\,M\ \subseteq\ X_i$\quad if $i$ is odd,
\end{minipage}\end{equation}
while
\begin{equation}\begin{minipage}[c]{23pc}\label{d.even}
$X_i\,M\ \subseteq\ X_{i-1}\cup X_i \cup X_{i+1}$\quad if $i$ is even.
\end{minipage}\end{equation}

To see~\eqref{d.odd}, let $i$ be odd and consider any element
$x = g_1\,g_2^{-1}\dots\,g_i \in X_i$
$(g_1,\,g_2,\dots,\,g_i\in M)$, and any $h\in M$.
Clearly, $xh$ still belongs to the \mbox{$i$-fold} product
$M\,M^{-1}\dots\,M$
so we need only show that we cannot write
$xh = g'_1\,{g'_2}^{-1}\dots\,{g'_{i-1}}^{\hspace*{-.8em}-1}$
with all $g'_j \in M$.
And indeed, if we could, we would have
$x = g'_1\,{g'_2}^{-1}\dots\,(h\,g'_{i-1})^{-1}$,
an expression as a length\,-\,$i{-}1$ product,
contradicting our assumption that $x\in X_i$.
The assertion~\eqref{d.even} is verified similarly:
the product of an element $x\in X_i$ and an element $h\in M$
can be written as an $i{+}1$-fold alternating product
of members of $M$ and $M^{-1}$, and if we could write it
as a product of the same sort having fewer than $i-1$ terms,
this would lead to a representation of $x$ having fewer than
$i$ terms, again contradicting the condition~$x\in X_i$.

Clearly,~\eqref{d.odd} and~\eqref{d.even}
have the same form as the restrictions pictured in~\eqref{d.u_d}.

Of course, this similarity
does not extend to the assertion that elements of $G$ move
only finitely many members of $X$ from one $X_i$ to another.

\section{Acknowledgements}\label{S.ackn}

I am indebted to Ralph McKenzie for a fruitful correspondence
on the 1981 draft of this note, to Charles Pugh for the
measure-theoretic result quoted in Section~\ref{SS.small/big},
to Dawid Kielak for the question answered in \S\ref{SS.fg},
and to the referee for several helpful comments.

\end{document}